\documentclass[a4paper,twoside]{article}
\usepackage{a4}
\usepackage{amssymb}
\usepackage{amsmath}
\usepackage{upref}
\usepackage[colorlinks,citecolor=blue,linkcolor=blue]{hyperref}
\usepackage[dvipsnames]{color}
\usepackage[active]{srcltx}
\allowdisplaybreaks[2] 
\newcount\minutes \newcount\hours
\hours=\time \divide\hours 60 \minutes=\hours
 \multiply\minutes-60
\advance\minutes \time
\newcommand{\klockan}{\the\hours:{\ifnum\minutes<10 0\fi}\the\minutes}
\newcommand{\tid}{\today\ \klockan}
\newcommand{\prtid}{\smash{\raise 10mm \hbox{\LaTeX ed \tid}}}
\renewcommand{\prtid}{}
\makeatletter  \headheight 10pt
\def\sectionmark#1{}
\def\subsectionmark#1{}
\newcommand{\sectnr}{\ifnum \c@secnumdepth >\z@
                 \thesection.\hskip 1em\relax \fi}
\def\@evenhead{\footnotesize\rm\thepage\hfil\leftmark\hfil\llap{\prtid}}
\def\@oddhead{\footnotesize\rm\rlap{\prtid}\hfil\rightmark\hfil\thepage}
\def\tableofcontents{\section*{Contents}
 \@starttoc{toc}}
\makeatother
\makeatletter
\def\@biblabel#1{#1.}
\makeatother
\makeatletter
\let\Thebibliography=\thebibliography
\renewcommand{\thebibliography}[1]{\def\@mkboth##1##2{}\Thebibliography{#1}
\addcontentsline{toc}{section}{References}
\frenchspacing 
\setlength{\@topsep}{0pt}
\setlength{\itemsep}{0pt}%
\setlength{\parskip}{0pt plus 2pt}%
} \makeatother
\makeatletter
\def\mdots@{\mathinner.\nonscript\!.%
 \ifx\next,.\else\ifx\next;.\else\ifx\next..\else
 \nonscript\!\mathinner.\fi\fi\fi}
\let\ldots\mdots@
\makeatother
\makeatletter
\let\Enumerate=\enumerate
\renewcommand{\enumerate}{\Enumerate%
\setlength{\@topsep}{0pt}
\setlength{\itemsep}{0pt}%
\setlength{\parskip}{0pt plus 1pt}%
\renewcommand{\theenumi}{\textup{(\alph{enumi})}}%
\renewcommand{\labelenumi}{\theenumi}%
}
\let\endEnumerate=\endenumerate
\renewcommand{\endenumerate}{\endEnumerate\unskip}
\makeatother
\makeatletter
\def\@seccntformat#1{\csname the#1\endcsname.\quad}
\makeatother
\newcommand{\authortitle}[2]{\author{#1}\title{#2}\markboth{#1}{#2}}
\newcommand{\art}[6]{{\sc #1, \rm #2, \it #3 \bf #4 \rm (#5), \mbox{#6}.}}
\newcommand{\auth}[2]{{#1, #2.}}
\newcommand{\idxauth}[2]{{#1, #2.}}
\newcommand{\artin}[3]{{\sc #1, \rm #2,  in #3.}}
\newcommand{\artprep}[3]{{\sc #1, \rm #2, #3.}}
\newcommand{\arttoappear}[3]{{\sc #1, \rm #2, to appear in \it #3}}
\newcommand{\book}[3]{{\sc #1, \it #2, \rm #3.}}
\newcommand{\AND}{{\rm and }}
\RequirePackage{amsthm}
\newtheoremstyle{descriptive}%
  {\topsep}   
  {\topsep}   
  {\rmfamily} 
  {}          
  {\bfseries} 
  {.}         
  { }         
  {}          
\newtheoremstyle{propositional}%
  {\topsep}   
  {\topsep}   
  {\itshape}  
  {}          
  {\bfseries} 
  {.}         
  { }         
  {}          
\theoremstyle{propositional}
\newtheorem{thm}{Theorem}[section]
\newtheorem{prop}[thm]{Proposition}
\newtheorem{lem}[thm]{Lemma}
\newtheorem{lemma}[thm]{Lemma}
\newtheorem{cor}[thm]{Corollary}
\theoremstyle{descriptive}
\newtheorem{deff}[thm]{Definition}
\newtheorem{example}[thm]{Example}
\newtheorem{remark}[thm]{Remark}
\makeatletter
\renewenvironment{proof}[1][\proofname]{\par
  \pushQED{\qed}%
  \normalfont
  \trivlist
  \item[\hskip\labelsep
        \itshape
    #1\@addpunct{.}]\ignorespaces
}{%
  \popQED\endtrivlist\@endpefalse
} \makeatother
\newcommand{\setm}{\setminus}
\renewcommand{\emptyset}{\varnothing}
\def\vint{\mathop{\mathchoice%
          {\setbox0\hbox{$\displaystyle\intop$}\kern 0.22\wd0%
           \vcenter{\hrule width 0.6\wd0}\kern -0.82\wd0}%
          {\setbox0\hbox{$\textstyle\intop$}\kern 0.2\wd0%
           \vcenter{\hrule width 0.6\wd0}\kern -0.8\wd0}%
          {\setbox0\hbox{$\scriptstyle\intop$}\kern 0.2\wd0%
           \vcenter{\hrule width 0.6\wd0}\kern -0.8\wd0}%
          {\setbox0\hbox{$\scriptscriptstyle\intop$}\kern 0.2\wd0%
           \vcenter{\hrule width 0.6\wd0}\kern -0.8\wd0}}%
          \mathopen{}\int}
\newcommand{\Cp}{{C_p}}
\DeclareMathOperator{\capp}{cap}
\DeclareMathOperator{\Chcapp}{Ch-cap}
\newcommand{\cp}{\capp_p}
\newcommand{\cpp}{\Chcapp_p}
\DeclareMathOperator{\diam}{diam}
\DeclareMathOperator{\dist}{dist}
\DeclareMathOperator{\Lip}{Lip}
\newcommand{\Lipc}{{\Lip_c}}
\DeclareMathOperator{\spt}{supp}
\newcommand{\supp}{\spt}

\DeclareMathOperator*{\essinf}{ess\,inf}
\DeclareMathOperator*{\essliminf}{ess\,lim\,inf}
\DeclareMathOperator{\fineint}{fine-int}

\newcommand{\bdry}{\partial}
\newcommand{\bdy}{\bdry}
\newcommand{\loc}{_{\rm loc}}
{\catcode`p =12 \catcode`t =12 \gdef\eeaa#1pt{#1}}      
\def\accentadjtext#1{\setbox0\hbox{$#1$}\kern   
                \expandafter\eeaa\the\fontdimen1\textfont1 \ht0 }
\def\accentadjscript#1{\setbox0\hbox{$#1$}\kern 
                \expandafter\eeaa\the\fontdimen1\scriptfont1 \ht0 }
\def\accentadjscriptscript#1{\setbox0\hbox{$#1$}\kern   
                \expandafter\eeaa\the\fontdimen1\scriptscriptfont1 \ht0 }
\def\accentadjtextback#1{\setbox0\hbox{$#1$}\kern       
                -\expandafter\eeaa\the\fontdimen1\textfont1 \ht0 }
\def\accentadjscriptback#1{\setbox0\hbox{$#1$}\kern     
                -\expandafter\eeaa\the\fontdimen1\scriptfont1 \ht0 }
\def\accentadjscriptscriptback#1{\setbox0\hbox{$#1$}\kern 
                -\expandafter\eeaa\the\fontdimen1\scriptscriptfont1 \ht0 }
\def\itoverline#1{{\mathsurround0pt\mathchoice
        {\rlap{$\accentadjtext{\displaystyle #1}
                \accentadjtext{\vrule height1.593pt}
                \overline{\phantom{\displaystyle #1}
                \accentadjtextback{\displaystyle #1}}$}{#1}}
        {\rlap{$\accentadjtext{\textstyle #1}
                \accentadjtext{\vrule height1.593pt}
                \overline{\phantom{\textstyle #1}
                \accentadjtextback{\textstyle #1}}$}{#1}}
        {\rlap{$\accentadjscript{\scriptstyle #1}
                \accentadjscript{\vrule height1.593pt}
                \overline{\phantom{\scriptstyle #1}
                \accentadjscriptback{\scriptstyle #1}}$}{#1}}
        {\rlap{$\accentadjscriptscript{\scriptscriptstyle #1}
                \accentadjscriptscript{\vrule height1.593pt}
                \overline{\phantom{\scriptscriptstyle #1}
                \accentadjscriptscriptback{\scriptscriptstyle #1}}$}{#1}}}}
\newcommand{\al}{\alpha}
\newcommand{\alp}{\alpha}
\newcommand{\dmu}{d\mu}
\newcommand{\eps}{\varepsilon}
\newcommand{\ga}{\gamma}
\newcommand{\Om}{\Omega}
\renewcommand{\phi}{\varphi}
\newcommand{\p}{{$p\mspace{1mu}$}}
\newcommand{\clEp}{{\itoverline{E}\mspace{1mu}}^p}
\newcommand{\bdyp}{\bdy_p} 
\newcommand{\R}{\mathbf{R}}
\newcommand{\Q}{\mathbf{Q}}
\newcommand{\eR}{{\overline{\R}}}
\newcommand{\K}{{\cal K}}
\newcommand{\oHpind}[1]{H_{#1}}     
\newcommand{\limplus}{{\mathchoice{\raise.17ex\hbox{$\scriptstyle +$}}
		{\raise.17ex\hbox{$\scriptstyle +$}}
		{\raise.1ex\hbox{$\scriptscriptstyle +$}}
		{\scriptscriptstyle +}}}
\newcommand{\limminus}{{\mathchoice{\raise.17ex\hbox{$\scriptstyle -$}}
		{\raise.17ex\hbox{$\scriptstyle -$}}
		{\raise.1ex\hbox{$\scriptscriptstyle -$}}
		{\scriptscriptstyle -}}}
\newcommand{\Np}{N^{1,p}}
\newcommand{\Nploc}{N^{1,p}\loc}
\newcommand{\Dp}{D^p}
\newcommand{\Dploc}{D^{p}\loc}
\newcommand{\Ga}{\Gamma}
\newcommand{\Lploc}{L^p\loc}
\newcommand{\imp}{\ensuremath{\Rightarrow} }
\makeatletter
\newcommand{\setcurrentlabel}[1]{\def\@currentlabel{#1}}
\makeatother
\numberwithin{equation}{section}

\newenvironment{ack}{\medskip{\it Acknowledgement.}}{}

\begin{document}

\authortitle{Anders Bj\"orn, Jana Bj\"orn and Visa Latvala}
{The Cartan, Choquet and Kellogg properties for the
fine topology on metric spaces}
\author{
Anders Bj\"orn \\
\it\small Department of Mathematics, Link\"opings universitet, \\
\it\small SE-581 83 Link\"oping, Sweden\/{\rm ;}
\it \small anders.bjorn@liu.se
\\
\\
Jana Bj\"orn \\
\it\small Department of Mathematics, Link\"opings universitet, \\
\it\small SE-581 83 Link\"oping, Sweden\/{\rm ;}
\it \small jana.bjorn@liu.se
\\
\\
Visa Latvala \\
\it\small Department of Physics and Mathematics,
University of Eastern Finland, \\
\it\small P.O. Box 111, FI-80101 Joensuu,
Finland\/{\rm ;}
\it \small visa.latvala@uef.fi
}

\date{}

\maketitle

\noindent{\small
 {\bf Abstract}. We prove the Cartan and Choquet properties for the fine topology
on a complete metric space equipped with a doubling measure
supporting a \p-Poincar\'e inequality, $1 < p< \infty$. We apply these
key tools to establish a fine version of the Kellogg property,
characterize finely continuous functions by means of quasicontinuous
functions, and show that capacitary measures
associated with Cheeger supersolutions are supported by the fine boundary of the set.}

\bigskip

\noindent {\small \emph{Key words and phrases}:
Capacitary measure,
Cartan property, Choquet property,
doubling measure,
fine Kellogg property,
finely continuous, finely open, fine topology,
metric space,
nonlinear potential theory,
\p-harmonic,
Poincar\'e inequality,
quasicontinuous, quasiopen,
thin.
}
\medskip


\section{Introduction}

The aim of this paper is to
establish the Cartan and Choquet properties for the fine topology
on a complete metric space $X$ equipped with a doubling measure $\mu$
supporting a \p-Poincar\'e inequality, $1 < p< \infty$.
These properties are crucial for deep applications of the fine topology
in potential theory.
As applications of these key tools we establish
the fine Kellogg property and characterize finely continuous functions by means of
quasicontinuous functions.
We also show that capacitary measures associated with
Cheeger \p-supersolutions are supported by the fine boundary of the set
(not just by the metric boundary).

The classical fine topology is closely related to the Dirichlet
problem for the Laplace equation. Wiener~\cite{wiener} showed in 1924 that a
boundary point of a domain  is irregular if and only if
the complement  is thin at that point in a certain capacity
density sense, cf.\ Definition~\ref{deff-thinness}.
In 1939 Brelot~\cite{brelot39B}, \cite{brelot40} characterized thinness by a condition,
which is nowadays called the Cartan property.
The reason for this name is
that Cartan (in a letter to Brelot in 1940, see Brelot~\cite[p.~14]{brelot44})
connected the notion of thinness to the coarsest
topology making all superharmonic functions continuous.
Cartan~\cite{cartan46} coined
the name fine topology for such a topology.

Nonlinear potential theory associated with \p-harmonic functions
has been studied since the 1960s.
For extensive treatises and notes on the history,
see the monographs Adams--Hedberg~\cite{AdHe},
Heinonen--Kilpel\"ainen--Martio~\cite{HeKiMa} and
Mal\'y--Ziemer~\cite{MZ}.
Starting in the 1990s a lot of attention
has been given to analysis on metric spaces,
see e.g.\
Haj\l asz--Koskela~\cite{HaKo}, Heinonen~\cite{heinonen}
and Heinonen--Koskela~\cite{HeKo98}.
Around 2000 this initiated
studies of \p-harmonic and \p-superharmonic functions on
metric spaces without a differentiable structure,
see e.g.\
Bj\"orn--Bj\"orn~\cite{BBbook},
Bj\"orn--Bj\"orn--Shan\-mu\-ga\-lin\-gam~\cite{BBS},
Kinnunen--Martio~\cite{KiMa02},
Kinnunen--Shan\-mu\-ga\-lin\-gam~\cite{KiSh01} and
Shan\-mu\-ga\-lin\-gam~\cite{Sh-harm}.

The classical linear fine potential theory and fine topology
(the case $p=2$) have been
systematically studied since the 1960s.
Let us here just mention
Brelot~\cite{brelot71},
Fuglede~\cite{Fugl71}, \cite{Fug} and
Luke\v{s}--Mal\'y--Zaj\'i\v{c}ek~\cite{LuMaZa}, which include most of
the theory and the main references.
Some of these works are written in large generality including
topological spaces, general capacities and families of functions, and some results
thus apply also to the nonlinear theory.
At the same time, many other results rely indirectly on a linear structure, e.g.\
through potentials, integral representations and convex cones of superharmonic
functions, which are in general not available in the nonlinear setting.

The nonlinear fine potential theory started in
the 1970s on unweighted $\R^n$, see the notes to Chapter~12 in
\cite{HeKiMa} and Section~2.6 in \cite{MZ}.
For the fine potential theory associated with \p-harmonic functions
on unweighted $\R^n$, see \cite{MZ} and Latvala~\cite{LatPhD}.
The monograph \cite{HeKiMa}
is the main source for fine potential theory on weighted $\R^n$
(note that Chapter~21, which is only in the second addition, contains
some more recent results).
The study of fine potential theory on metric spaces is more recent,
see e.g.\
Bj\"orn--Bj\"orn~\cite{BBnonopen},
Bj\"orn--Bj\"orn--Latvala~\cite{BBL1},
J.~Bj\"orn~\cite{JB-pfine}, Kinnunen--Latvala~\cite{KiLa}
and Korte~\cite{korte08}.
For further references
to nonlinear and fine nonlinear potential theory,
see the introduction to \cite{BBL1}.

Recently, in \cite{BBL1}, we established the
so-called  \emph{weak Cartan property}, which
says that if $E \subset X$ is thin at $x_0 \notin E$, then there exist
a ball $B \ni x_0$ and superharmonic functions $u,u'$ on $B$ such that
\[
   v(x_0)<\liminf_{E\ni x\to x_0} v(x),
\]
where $v=\max\{u,u'\}$.

The superharmonic functions considered in~\cite{BBL1} were based on upper gradients,
and because of the lack of a differential equation, we did not succeed
in obtaining the full Cartan property
as in $\R^n$, where $v$ itself can be chosen superharmonic, cf.\
Theorem~\ref{thm-cartan} below.
Indeed, the proof of the full Cartan property seems to be as hard
as the proof of the Wiener criterion, which is also
open in the nonlinear potential theory based on upper gradients,
but is known to hold in the potential theory based on Cheeger gradients,
see J.~Bj\"orn~\cite{JB-Matsue}.
Nevertheless,
the weak Cartan property in~\cite{BBL1} was enough to conclude
that the fine topology
is the coarsest one making all superharmonic functions continuous.

Here, we instead focus on Cheeger superharmonic functions based on Cheeger's
theorem yielding a vector-valued Cheeger gradient.
In this case we do have an equation available and this enables us
to establish the following full Cartan property.

\begin{thm}\label{thm-cartan}
\textup{(Cartan property)}
Suppose that $E$ is thin at $x_0\in\itoverline{E}\setminus E$.
Then there is a bounded positive Cheeger superharmonic function $u$
in an open neighbourhood of $x_0$ such that
\[
u(x_0)<\liminf_{E\ni x\to x_0}u(x).
\]
\end{thm}

For a Newtonian function, the minimal
\p-weak upper gradient  and the modulus of the Cheeger
gradient are comparable.
Thus the corresponding capacities are comparable to each other,
and the fine topology, as well as thinness (and thickness),
is the same in both cases.
Superminimizers, superharmonic and \p-harmonic functions
are however different.
Hence, using the Cheeger structure we can study
thinness and the fine topology, but not e.g.\
the superharmonic and \p-harmonic functions based on upper gradients.
Only Cheeger \p-(super)harmonic functions can be treated.

We use the Cartan property
to establish the following important Choquet property.

\begin{thm} \label{thm-choquet-intro}
\textup{(Choquet property)}
For any $E\subset X$ and any $\eps >0$ there is an open set $G$
containing all the points in $X$ at which $E$ is thin,
such that $\Cp(E \cap G)< \eps$.
\end{thm}

The Choquet property was first established by Choquet~\cite{Choquet1959} in
1959. In the nonlinear theory on unweighted $\R^n$  it was later
established by Hedberg~\cite{Hedb} and  Hedberg--Wolff~\cite{HedWol} in
connections with potentials (also for higher-order Sobolev spaces).
The Cartan property for \p-superharmonic functions on unweighted
$\R^n$ was obtained by Kilpel\"ainen--Mal\'y~\cite{KiMa}
as a consequence of their pointwise Wolff-potential estimates.
In fact, Kilpel\"ainen and Mal\'y used the Cartan property to
establish the necessity in the Wiener criterion.
In Mal\'y--Ziemer~\cite{MZ}, the authors
deduce the Choquet property from the Cartan property.
The proof of the Cartan property was extended to weighted $\R^n$ by
Mikkonen~\cite[Theorem~5.8]{Mikkonen} and can also be found
 in Heinonen--Kilpel\"ainen--Martio~\cite[Theorem~21.26 (which is only in
the second edition)]{HeKiMa}; in both places
they however refrained from deducing consequences such as the Choquet
property.

Our proof of the Choquet property follows the one in \cite{MZ}, but we have some extra complications due to the fact that we
can simultaneously have some points with zero capacity and others with positive
capacity. Note that Fuglede~\cite{Fugl71} contains a proof of the Choquet property,
in an axiomatic setting, assuming that
Corollary~\ref{cor-kellogg-intro} and part (a) in
Theorem~\ref{thm-finelyopen-quasiopen} are true.
We have a converse approach, since our proofs of
Corollary~\ref{cor-kellogg-intro} and Theorem~\ref{thm-finelyopen-quasiopen}
are based on the Choquet property.

\begin{cor}\textup{(Fine Kellogg property)}\label{cor-kellogg-intro}
For any $E\subset X$ we have
\begin{equation}   \label{fine-Kellogg-intro}
C_p(\{x \in E: E \text{ is thin at } x\})=0.
\end{equation}
\end{cor}

The fine Kellogg property has close connections with boundary regularity,
see Remark~\ref{twokelloggs}.
The implications $\Rightarrow$ in the following result were already
obtained in Bj\"orn--Bj\"orn--Latvala~\cite{BBL1},
but now we are able to complete the picture.

\begin{thm} \label{thm-finelyopen-quasiopen}
{\rm(a)} A set $U\subset X$ is quasiopen if and only if $U=V \cup E$ for some finely open set $V$ and for a set $E$ of capacity zero.

{\rm(b)}  An extended real-valued function on a quasiopen set
$U$ is quasicontinuous in $U$ if and only if
$u$ is finite q.e.\ and finely continuous q.e.\ in $U$.
\end{thm}

It is pointed out in Adams--Lewis~\cite[Proposition~3]{AdLew} that (a) for unweighted $\R^n$ follows from the Choquet property established in Hedberg--Wolff~\cite{HedWol}.
Also (b) then follows by modifying the earlier axiomatic argumentation of
Fuglede~\cite[Lemma, p.\ 143]{Fugl71}. The proof of Theorem~\ref{thm-finelyopen-quasiopen} in unweighted  $\R^n$ is given in Mal\'y--Ziemer~\cite[p.~146]{MZ}.
For the reader's convenience, we include the proof of
Theorem~\ref{thm-finelyopen-quasiopen} although the proof essentially follows \cite{MZ}.
In Section~\ref{sect-finely-quasi}, we use Theorem~\ref{thm-finelyopen-quasiopen}
to extend and simplify some recent results from
Bj\"orn--Bj\"orn~\cite{BBnonopen}.

We end the paper with another application of the Cartan property
in Section~\ref{sect-supp-cap-meas}, which
contains
results on capacitary measures related to Cheeger
supersolutions, see
Theorem~\ref{thm-bootstrap} and 
Corollaries~\ref{cor-bootstrap} and \ref{cor-bootstrap2}.
In particular, we show that the capacitary measure only charges the fine boundary
of the corresponding set.

\begin{ack}
The  first two authors were supported by the Swedish Research Council.
Part of this research was done during several visits
of the third author to Link\"opings universitet in
2012--2014 and
while all three authors
visited Institut Mittag-Leffler in the autumn of 2013.
We thank both institutions for their hospitality
and support.
\end{ack}

\section{Notation and preliminaries}
\label{sect-prelim}

We assume throughout the paper that $1 < p<\infty$
and that $X=(X,d,\mu)$ is a metric space equipped
with a metric $d$ and a positive complete  Borel  measure $\mu$
such that $0<\mu(B)<\infty$ for all (open) balls $B \subset X$.
It follows that $X$ is separable.
The $\sigma$-algebra on which $\mu$ is defined
is obtained by the completion of the Borel $\sigma$-algebra.
We also assume that $\Om \subset X$ is a nonempty open
set.

We say that $\mu$  is \emph{doubling} if
there exists a \emph{doubling constant} $C>0$ such that for all balls
$B=B(x_0,r):=\{x\in X: d(x,x_0)<r\}$ in~$X$,
\begin{equation*}
        0 < \mu(2B) \le C \mu(B) < \infty.
\end{equation*}
Here and elsewhere we let $\delta B=B(x_0,\delta r)$.
A metric space with a doubling measure is proper\/
\textup{(}i.e.\ closed and bounded subsets are compact\/\textup{)}
if and only if it is complete.
See Heinonen~\cite{heinonen} for more on doubling measures.

A \emph{curve} is a continuous mapping from an interval,
and a \emph{rectifiable} curve is a curve with finite length.
We will only consider curves which are nonconstant, compact and rectifiable.
A curve can thus be parameterized by its arc length $ds$.
We follow Heinonen and Koskela~\cite{HeKo98} in introducing
upper gradients as follows (they called them very weak gradients).

\begin{deff} \label{deff-ug}
A nonnegative Borel function $g$ on $X$ is an \emph{upper gradient}
of an extended real-valued function $f$
on $X$ if for all nonconstant, compact and rectifiable curves
$\gamma: [0,l_{\gamma}] \to X$,
\begin{equation} \label{ug-cond}
        |f(\gamma(0)) - f(\gamma(l_{\gamma}))| \le \int_{\gamma} g\,ds,
\end{equation}
where we follow the convention that the left-hand side is $\infty$
whenever at least one of the
terms therein is infinite.
If $g$ is a nonnegative measurable function on $X$
and if (\ref{ug-cond}) holds for \p-almost every curve (see below),
then $g$ is a \emph{\p-weak upper gradient} of~$f$.
\end{deff}

Here we say that a property holds for \emph{\p-almost every curve}
if it fails only for a curve family $\Ga$ with zero \p-modulus,
i.e.\ there exists $0\le\rho\in L^p(X)$ such that
$\int_\ga \rho\,ds=\infty$ for every curve $\ga\in\Ga$.
Note that a \p-weak upper gradient need not be a Borel function,
it is only required to be measurable.
On the other hand,
every measurable function $g$ can be modified on a set of measure zero
to obtain a Borel function, from which it follows that
$\int_{\gamma} g\,ds$ is defined (with a value in $[0,\infty]$) for
\p-almost every curve $\ga$.
For proofs of these and all other facts in this section
we refer to Bj\"orn--Bj\"orn~\cite{BBbook} and
Heinonen--Koskela--Shanmugalingam--Tyson~\cite{HKST}.

The \p-weak upper gradients were introduced in
Koskela--MacManus~\cite{KoMc}. It was also shown there
that if $g \in \Lploc(X)$ is a \p-weak upper gradient of $f$,
then one can find a sequence $\{g_j\}_{j=1}^\infty$
of upper gradients of $f$ such that $g_j-g \to 0$ in $L^p(X)$.
If $f$ has an upper gradient in $\Lploc(X)$, then
it has a \emph{minimal \p-weak upper gradient} $g_f \in \Lploc(X)$
in the sense that for every \p-weak upper gradient $g \in \Lploc(X)$ of $f$ we have
$g_f \le g$ a.e., see Shan\-mu\-ga\-lin\-gam~\cite{Sh-harm}.
The minimal \p-weak upper gradient is well defined
up to a set of measure zero in the cone of nonnegative functions in $\Lploc(X)$.
Following Shanmugalingam~\cite{Sh-rev},
we define a version of Sobolev spaces on the metric measure space $X$.

\begin{deff} \label{deff-Np}
Let for measurable $f$,
\[
        \|f\|_{\Np(X)} = \biggl( \int_X |f|^p \, \dmu
                + \inf_g  \int_X g^p \, \dmu \biggr)^{1/p},
\]
where the infimum is taken over all upper gradients of $f$.
The \emph{Newtonian space} on $X$ is
\[
        \Np (X) = \{f: \|f\|_{\Np(X)} <\infty \}.
\]
\end{deff}
\medskip

The space $\Np(X)/{\sim}$, where  $f \sim h$ if and only if $\|f-h\|_{\Np(X)}=0$,
is a Banach space and a lattice, see Shan\-mu\-ga\-lin\-gam~\cite{Sh-rev}.
In this paper we assume that functions in $\Np(X)$
are defined everywhere (with values in $\eR:=[-\infty,\infty]$),
not just up to an equivalence class in the corresponding function space.
For a measurable set $E\subset X$, the Newtonian space $\Np(E)$ is defined by
considering $(E,d|_E,\mu|_E)$ as a metric space on its own.
We say  that $f \in \Nploc(E)$ if
for every $x \in E$ there exists a ball $B_x\ni x$ such that
$f \in \Np(B_x \cap E)$.
If $f,h \in \Nploc(X)$,
then $g_f=g_h$ a.e.\ in $\{x \in X : f(x)=h(x)\}$,
in particular $g_{\min\{f,c\}}=g_f \chi_{\{f < c\}}$ for $c \in \R$.

\begin{deff} \label{deff-sobcap}
The \emph{Sobolev capacity} of an arbitrary set $E\subset X$ is
\[
\Cp(E) = \inf_u\|u\|_{\Np(X)}^p,
\]
where the infimum is taken over all $u \in \Np(X)$ such that
$u\geq 1$ on $E$.
\end{deff}

The Sobolev capacity is countably subadditive.
We say that a property holds \emph{quasieverywhere} (q.e.)\
if the set of points  for which the property does not hold
has Sobolev capacity zero.
The Sobolev capacity is the correct gauge
for distinguishing between two Newtonian functions.
If $u \in \Np(X)$, then $u \sim v$ if and only if $u=v$ q.e.
Moreover, Corollary~3.3 in Shan\-mu\-ga\-lin\-gam~\cite{Sh-rev} shows that
if $u,v \in \Np(X)$ and $u= v$ a.e., then $u=v$ q.e.

A set $U\subset X$ is \emph{quasiopen} if for every
$\varepsilon>0$ there is an open set $G\subset X$ such that $\Cp(G)<\varepsilon$
and $G\cup U$ is open.
A function $u$ defined on a set $E\subset X$ is \emph{quasicontinuous}
if for every $\varepsilon>0$ there is an open set $G\subset X$ such that $\Cp(G)<\varepsilon$ and $u|_{E \setm G}$ is finite and continuous.
If $u$ is quasicontinuous on a quasiopen set $U$, then
it is easily verified that $\{x \in U: u(x) <a\}$ and $\{x \in U: u(x) >a\}$ are
quasiopen for all $a \in \R$, cf.\ Proposition~3.3 in
Bj\"orn--Bj\"orn--Mal\'y~\cite{BBM}.

\begin{deff} \label{def.PI.}
We say that $X$ supports a \emph{\p-Poincar\'e inequality} if
there exist constants $C>0$ and $\lambda \ge 1$
such that for all balls $B \subset X$,
all integrable functions $f$ on $X$ and all upper gradients $g$ of $f$,
\begin{equation} \label{PI-ineq}
        \vint_{B} |f-f_B| \,\dmu
        \le C \diam(B) \biggl( \vint_{\lambda B} g^{p} \,\dmu \biggr)^{1/p},
\end{equation}
where $ f_B
 :=\vint_B f \,\dmu
:= \int_B f\, d\mu/\mu(B)$.
\end{deff}

In the definition of Poincar\'e inequality we can equivalently assume
that $g$ is a \p-weak upper gradient.

In $\R^n$ equipped with a doubling measure $d\mu=w\,dx$, where
$dx$ denotes Lebesgue measure, the \p-Poincar\'e inequality~\eqref{PI-ineq}
is equivalent to the \emph{\p-admissibility} of the weight $w$ in the
sense of Heinonen--Kilpel\"ainen--Martio~\cite{HeKiMa}, cf.\
Corollary~20.9 in~\cite{HeKiMa}
and Proposition~A.17 in~\cite{BBbook}.

If $X$ is complete and  supports a \p-Poincar\'e inequality
and $\mu$ is doubling, then Lipschitz functions
are dense in $\Np(X)$, see Shan\-mu\-ga\-lin\-gam~\cite{Sh-rev}.
Moreover, all functions in $\Np(X)$
and those in $\Np(\Om)$ are quasicontinuous,
see Bj\"orn--Bj\"orn--Shan\-mu\-ga\-lin\-gam~\cite{BBS5}.
This means that in the Euclidean setting, $\Np(\R^n)$ is the
refined Sobolev space as defined in
Heinonen--Kilpel\"ainen--Martio~\cite[p.~96]{HeKiMa},
see Bj\"orn--Bj\"orn~\cite[Appendix~A.2]{BBbook}
for a proof of this fact valid in weighted $\R^n$.
This is the main reason why, unlike in the classical Euclidean setting,
we do not need to
require the functions competing in the
definitions of capacity to be $1$ in a
neighbourhood of $E$.
For recent related progress on the density of Lipschitz functions
see Ambrosio--Colombo--Di Marino~\cite{AmbCD} and
Ambrosio--Gigli--Savar\'e~\cite{AmbGS}.

In Section~\ref{sect-fine-cont} the fine topology is
defined by means of thin sets, which in turn use the
variational capacity $\cp$.
To be able to define the variational capacity we first
need a Newtonian space with zero boundary values.
We let, for an arbitrary set $A \subset X$,
\[
\Np_0(A)=\{f|_{A} : f \in \Np(X) \text{ and }
        f=0 \text{ on } X \setm A\}.
\]
One can replace the assumption ``$f=0$ on $X \setm A$''
with ``$f=0$ q.e.\ on $X \setm A$''
without changing the obtained space
$\Np_0(A)$.
Functions from $\Np_0(A)$ can be extended by zero in $X\setm A$ and we
will regard them in that sense if needed.

\begin{deff}
The \emph{variational
capacity} of $E\subset \Om$ with respect to $\Om$ is
\[
\cp(E,\Om) = \inf_u\int_X g_u^p\, d\mu,
\]
where the infimum is taken over all $u \in \Np_0(\Om)$
such that
$u\geq 1$ on $E$.
\end{deff}

If $\Cp(E)=0$, then $\cp(E,\Om)=0$. The converse implication is true if $\mu$ is doubling and supports a \p-Poincar\'e inequality.

In Section~\ref{sect-supp-cap-meas} we will need the following
simple lemma.
For the reader's convenience we provide the short proof.

\begin{lem} \label{lem-product}
If $u,v \in \Np(X)$ are bounded,
then $uv \in \Np(X)$.
\end{lem}

\begin{proof} We can assume that $|u|$ and $|v|$ are bounded by $1$.
Then $|uv| \le |u|$ and hence $uv \in L^p(X)$.
By the Leibniz rule (Theorem~2.15 in Bj\"orn--Bj\"orn~\cite{BBbook}),
$    g:=|u|g_v + |v|g_u$ is a \p-weak upper gradient of $uv$.
As $g \le g_v +g_u \in L^p(X)$ we see that $uv \in \Np(X)$.
\end{proof}

Throughout the paper, the letter $C$ will denote various positive
constants whose values may vary even within a line.
We also write $A\simeq B$ if $C^{-1}A \le B \le C A$.

\section{Cheeger gradients}

\emph{Throughout the rest of the paper, we assume that
$X$ is complete and  supports a \p-Poincar\'e inequality, and that
$\mu$ is doubling.}

\medskip

In addition to upper gradients we will also use Cheeger gradients in this paper.
Their existence is based on the following deep
result of Cheeger.

\begin{thm} \label{Cheeger-thm}
\textup{(Theorem~4.38 in Cheeger~\cite{Cheeg})}
There exists $N$ and a countable collection\/ $(U_{\al}, X^{\al})$
of pairwise disjoint measurable sets
$U_{\al}$ and Lipschitz ``coordinate'' functions
$X^{\al}\colon  X \to \R^{k(\al)}$\textup{,}
$1\le k(\al) \le N$\textup{,} such that
$\mu\bigl( X\setminus \bigcup_{\al} U_{\al} \bigr) = 0$ and for every
Lipschitz function
$f\colon X\to \R$ there exist unique bounded vector-valued functions
$d ^{\al} f\colon U_{\al} \to \R^{k(\al)}$ such that for
a.e.\ $x\in U_{\al}$\textup{,}
\begin{equation*}
        \lim_{r\to0} \sup_{y\in B(x,r)}
               \frac{|f(y)-f(x) -  d^{\al} f(x) \cdot
                (X^{\al}(y) - X^{\al}(x))|}{r} =0,
\end{equation*}
where\/ $\cdot$ denotes the usual inner
product in\/ $\R^{k(\al)}$.
\end{thm}

Cheeger further shows that for a.e.\ $x\in U_{\al}$,
there is an inner product norm $|\cdot|_x$ on
$\R^{k(\al)}$ such that for all Lipschitz $f$,
\begin{equation}
\frac{1}{C}g_f(x) \le |d^{\al} f(x)|_x \le C g_f(x),
       \label{df-comp-gf}
\end{equation}
where $C$ is independent of $f$ and $x$, see p.\ 460 in
Cheeger~\cite{Cheeg}.
As Lipschitz functions are dense in $\Np(X)$,
the ``gradients'' $d^\al f$
extend uniquely  to the whole  $\Np(X)$,
by
Franchi--Haj\l asz--Koskela~\cite[Theorem~10]{FHK} or Keith~\cite{Ke-meas}.
Moreover, \eqref{df-comp-gf} holds even for functions in $\Np(X)$.

From now on we drop $\alp$ and set
\[
     D f := d^{\alp} f \quad \text{in } U_\alp.
\]

On a metric space there is some freedom in choosing the Cheeger
structure.
On $\R^n$ we will however always make the natural choice
$Df = \nabla f$ and let the inner product norm in \eqref{df-comp-gf}
be the Euclidean norm. Here $\nabla f$ denotes the Sobolev gradient from
Hei\-no\-nen--Kil\-pe\-l\"ai\-nen--Martio~\cite{HeKiMa},
which equals the distributional gradient if the weight on $\R^n$
is a Muckenhoupt $A_p$ weight. In this case, $|Df|=g_f$, by Proposition~A.13
in Bj\"orn--Bj\"orn~\cite{BBbook}.

\section{Supersolutions and superharmonic functions}
\label{sect-superharm}

In the literature on potential theory in metric spaces one usually studies
the following (super)minimizers based on upper gradients.

\begin{deff} \label{def-quasimin}
A function $u \in \Nploc(\Om)$ is a
\emph{\textup{(}super\/\textup{)}minimizer} in $\Om$
if
\[
      \int_{\{\phi \ne 0\}} g^p_u \, d\mu
           \le \int_{\{\phi \ne 0\}} g_{u+\phi}^p \, d\mu
           \quad \text{for all (nonnegative) } \phi \in \Lipc(\Om).
\]
A \emph{\p-harmonic function} is a continuous minimizer.
\end{deff}

Here $\Lipc(\Om)=\{\phi \in \Lip(X): \supp \phi \Subset \Om\}$
and $E \Subset \Om$ if $\itoverline{E}$ is a compact subset of $\Om$.

Minimizers were first studied by Shanmugalingam~\cite{Sh-harm},
and superminimizers by Kinnunen--Martio~\cite{KiMa02}.
For various characterizations of minimizers and superminimizers
  see A.~Bj\"orn~\cite{ABkellogg}.
If $u$ is a superminimizer, then
its \emph{lsc-regularization}
\begin{equation}
\label{essliminf}
 u^*(x):=\essliminf_{y\to x} u(y)= \lim_{r \to 0} \essinf_{B(x,r)} u
\end{equation}
is also a superminimizer  and $u^*= u$ q.e.,
see  \cite{KiMa02} or
Bj\"orn--Bj\"orn--Parviainen~\cite{BBP}.
If $u$ is a minimizer, then $u^*$ is  continuous,
(by Kinnunen--Shan\-mu\-ga\-lin\-gam~\cite{KiSh01}
or Bj\"orn--Marola~\cite{BMarola}),
and thus \p-harmonic.
For further discussion and references on the topics in this section
see \cite{BBbook}.

In this paper, we consider
\emph{Cheeger\/ \textup{(}super\/\textup{)}minimizers} and
\emph{Cheeger \p-harmonic functions} defined by replacing $g_u$ and $g_{u+\phi}$ in
Definition~\ref{def-quasimin} by $|Du|$ and $|D(u+\phi)|$, respectively,
where $|\cdot|$ is the inner product norm
in \eqref{df-comp-gf}.
Due to the  vector structure of the Cheeger gradient one can also
make the following definition. (There is no corresponding notion
for upper gradients.)

\begin{deff}
A function $u\in N^{1,p}\loc(\Omega)$ is a
\emph{\textup{(}super\/\textup{)}solution} in $\Omega$ if
\[
\int_{\Omega}|Du|^{p-2}Du\cdot D\varphi\,d\mu\ge 0
           \quad \text{for all (nonnegative) } \phi \in \Lipc(\Om),
\]
where $\cdot$ is the inner product giving rise to the norm
in \eqref{df-comp-gf}.
\end{deff}

It can be shown
that a function is a (super)solution if and only
if it is a Cheeger (super)minimizer, the proof is the same as for
Theorem~5.13 in
Hei\-no\-nen--Kil\-pe\-l\"ai\-nen--Martio~\cite{HeKiMa}.
In weighted $\R^n$, with the choice $Df = \nabla f$, we have
$g_f=|Df|=|\nabla f|$ a.e.\ which implies that
(super)minimizers, Cheeger (super)minimizers and (super)solutions
coincide, and are the same as in \cite{HeKiMa}.

Let $G \subset X$ be a nonempty bounded open set with $\Cp(X \setm G)>0$.
We consider the following obstacle problem in $G$.

\begin{deff} \label{deff-obst}
For $f \in \Np(G)$ and $\psi : G \to \eR$ let
\[
    \K_{\psi,f}(G)=\{v \in \Np(G) : v-f \in \Np_0(G)
            \text{ and } v \ge \psi \ \text{q.e. in } G\}.
\]
A function $u \in \K_{\psi,f}(G)$
is a \emph{solution of the $\K_{\psi,f}(G)$-Cheeger obstacle problem}
if
\[
       \int_G |D u|^p \, d\mu
       \le \int_G |D v|^p \, d\mu
       \quad \text{for all } v \in \K_{\psi,f}(G).
\]
\end{deff}

A solution to the $\K_{\psi,f}(G)$-Cheeger obstacle problem is easily
seen to be a Cheeger superminimizer
(i.e.\ a supersolution)
in $G$.
Conversely, a supersolution $u$ in $\Om$ is a solution
of the $\K_{u,u}(G)$-Cheeger obstacle problem for all open $G \Subset \Om$
with $\Cp(X\setm G)>0$.

If $\K_{\psi,f}(G) \ne \emptyset$, then
there is a solution $u$ of the $\K_{\psi,f}(G)$-Cheeger obstacle problem,
and this solution is unique up to equivalence in $\Np(G)$.
Moreover, $u^*$
is the unique lsc-regularized solution.
Conditions for when $\K_{\psi,f}(G) \ne \emptyset$ can be found in
Bj\"orn--Bj\"orn~\cite{BBnonopen}.
If the obstacle $\psi$ is continuous,
as a function with values in $[-\infty,\infty)$,
then $u^*$
is also continuous.
These results were obtained for the upper gradient obstacle problem
by Kinnunen--Martio~\cite{KiMa02}, where superharmonic functions based on upper gradients were also introduced.
As with most of the results in the metric theory their proofs
work verbatim for the Cheeger case considered here.
Since most of the theory has been developed in the setting
of upper gradients, we will often just refer to the upper gradient equivalents
of results for Cheeger (super)minimizers.

For $f \in \Np(G)$, we let $\oHpind{G} f$ denote the
continuous solution of the $\K_{-\infty,f}(G)$-Cheeger obstacle problem.
This function is Cheeger \p-harmonic in $G$ and has the same boundary values
(in the Sobolev sense) as $f$ on $\partial G$, and hence is also
called the solution of the (Cheeger)
Dirichlet problem with Sobolev boundary values.

\begin{deff} \label{deff-superharm}
A function $u : \Om \to (-\infty,\infty]$ is \emph{Cheeger superharmonic}
in $\Om$ if
\begin{enumerate}
\renewcommand{\theenumi}{\textup{(\roman{enumi})}}%
\item $u$ is lower semicontinuous;
\item \label{cond-b}
$u$ is not identically $\infty$ in any component of $\Om$;
\item \label{cond-c}
for every nonempty open set $\Om'\Subset \Om$ with
$\Cp(X \setm \Om')>0$
and all functions
$v \in \Lip(X)$,
we have $\oHpind{\Om'} v \le u$ in $\Om'$
whenever $v \le u$ on $\bdy \Om'$.
\end{enumerate}
\end{deff}

This definition of Cheeger superharmonicity is equivalent to  the one in
Hei\-no\-nen--Kil\-pe\-l\"ai\-nen--Martio~\cite{HeKiMa}, see A.~Bj\"orn~\cite{ABsuper}.
A locally bounded Cheeger superharmonic function is a supersolution,
and all Cheeger superharmonic functions are lsc-regularized.
Conversely, any lsc-regularized supersolution
is Cheeger superharmonic.
See Kinnunen--Martio~\cite{KiMa02}.

\begin{deff} \label{deff-Ch-pot}
The \emph{Cheeger capacitary potential} $u_E$ of a set $E \subset G$ in $G$
is the lsc-regularized solution of the $\K_{\chi_E,0}(G)$-Cheeger obstacle
 problem.

The \emph{Cheeger variational
capacity} of $E\subset G$ is defined as
\begin{equation} \label{eq-cpp-def}
\cpp(E,G) = \int_X |Du_E|^p\, d\mu = \inf_u\int_X |Du|^p\, d\mu,
\end{equation}
where the infimum is taken over all $u \in \Np_0(G)$
such that $u\geq 1$ on $E$.
\end{deff}

By \eqref{df-comp-gf}, we have
\begin{equation}  \label{eq-cpp-cp}
     \cpp(E,G) \simeq \cp(E,G).
\end{equation}

\section{Supersolutions and Radon measures}
\label{sect-supersoln-Radon}

\emph{In this section we assume that $\Om$ is a
nonempty bounded open set with $\Cp(X \setm \Om)>0$.}

\medskip

It was shown in
Bj\"orn--MacManus--Shanmugalingam~\cite[Propositions~3.5 and~3.9]{BMS}
that there is a one-to-one
correspondence between supersolutions
in $\Om$ and Radon measures in the
dual $\Np_0(\Om)'$.
A \emph{Radon measure} is a positive complete Borel measure which
is finite on every compact set. 

\begin{prop}    \label{prop-supersol-Radon}
For every supersolution $u$
in $\Om$ there is a Radon measure $\nu \in \Np_0(\Om)'$
such that for all $\phi \in \Np_0(\Om)$,
\begin{equation}  \label{eq-deff-Tu}
Tu(\phi):= \int_\Om|Du|^{p-2}Du\cdot D\phi\, d\mu = \int_\Om
\phi\ d\nu,
\end{equation}
where $\cdot$ is the inner product giving rise to the norm
in \eqref{df-comp-gf}.

Conversely, if $\nu \in \Np_0(\Om)'$ is a Radon measure on $\Om$ then
there exists a unique lsc-regularized
$u\in\Np_0(\Om)$ satisfying $Tu = \nu$ in the
sense of~\eqref{eq-deff-Tu} for all $\phi \in \Np_0(\Om)$.
Moreover, $u$ is a nonnegative supersolution in $\Om$.
\end{prop}

\begin{remark} This result is always false if we drop the assumption
$\Cp(X \setm \Om)>0$. Indeed, if $u$
is a nonnegative lsc-regularized supersolution in $\Om$,
then $u$ is Cheeger superharmonic in $\Om$.
If $\Cp(X \setm \Om)=0$, then
$u$ has a Cheeger superharmonic extension to $X$,
by Theorem~6.3  in A.~Bj\"orn~\cite{ABremove}
(or Theorem~12.3 in \cite{BBbook}),
which must be constant, by Corollary~9.14 in \cite{BBbook}.
(Note that if $\Om$ is bounded and $\Cp(X \setm \Om)=0$,
then also $X$ must be bounded.)
On the other hand, there are nonzero Radon measures
in $\Np_0(\Om)'$, so the existence
of a corresponding supersolution fails.
\end{remark}

\begin{proof}[Proof of Proposition~\ref{prop-supersol-Radon}]
See \cite[Propositions~3.5 and~3.9]{BMS},
where the result was stated under stronger assumptions than here, but the proof of this result is valid under our
assumptions. In particular,
as $\Cp(X \setm \Om)>0$,
the coercivity of the map $T$ follows from
the Poincar\'e inequality for $\Np_0$ (also called the \p-Friedrichs'
inequality), whose proof can be found e.g.\ in Corollary~5.54 in~\cite{BBbook}.

In \cite{BMS} the uniqueness was shown
up to equivalence between supersolutions.
The pointwise uniqueness for lsc-regularized supersolutions then
follows from~\eqref{essliminf}.
That $u$ is nonnegative follows from Lemma~\ref{lem-comp-potentials}
below, as $u \equiv 0$ is the
lsc-regularized supersolution corresponding to the zero measure.
\end{proof}

We need the following comparison principle.

\begin{lem}  \label{lem-comp-potentials}
Let
$\nu_1,\nu_2\in\Np_0(\Om)'$ be Radon measures
such that $\nu_1\le\nu_2$.
If $u_1,u_2\in\Np_0(\Om)$ are the corresponding
lsc-regularized
supersolutions given
by Proposition~\ref{prop-supersol-Radon}, then
$u_1\le u_2$
in $\Om$.
\end{lem}

\begin{proof}
Inserting $\phi=(u_1-u_2)_\limplus\in\Np_0(\Om)$ into
the equation~\eqref{eq-deff-Tu} for $u_1$ and $u_2$ gives
\begin{align}   \label{eq-test-with-phi}
0 &\le \int_\Om \phi\,d\nu_2 - \int_\Om \phi\,d\nu_1 \\
&= \int_\Om (|Du_2|^{p-2} Du_2 - |Du_1|^{p-2} Du_1)\cdot D\phi \,d\mu
\nonumber\\
&= \int_{\{u_1>u_2\}} (|Du_2|^{p-2} Du_2 - |Du_1|^{p-2} Du_1)\cdot (Du_1-Du_2) \,d\mu
\nonumber\\
&\le \int_{\{u_1>u_2\}} (|Du_2|^{p-1} |Du_1| + |Du_1|^{p-1} |Du_2|
- |Du_1|^p - |Du_2|^p) \,d\mu. \nonumber
\end{align}
The Young inequality shows that
\begin{align}  \label{eq-Holder-Young}
|Du_2|^{p-1} |Du_1| + |Du_1|^{p-1} |Du_2|
& \le \frac{p-1}{p} |Du_2|^p + \frac{1}{p} |Du_1|^p \nonumber \\
& \quad
+ \frac{p-1}{p} |Du_1|^p + \frac{1}{p} |Du_2|^p \nonumber\\
&
= |Du_1|^p + |Du_2|^p.
\end{align}
Inserting this into~\eqref{eq-test-with-phi} shows that equality must hold
in~\eqref{eq-test-with-phi} and~\eqref{eq-Holder-Young}
for a.e.\ $x$ such that  $u_1(x)>u_2(x)$.
This implies that $Du_1(x)=k(x)Du_2(x)$ for some $k(x)\ge0$ (by equality
in the last step of \eqref{eq-test-with-phi})
and $|Du_1(x)|=|Du_2(x)|$
(by equality in the Young inequality)
for a.e.\ $x$ with $u_1(x)>u_2(x)$.
It follows that $Du_1(x)=Du_2(x)$ for a.e.\ such $x$ and hence $D\phi=0$
a.e.\ in $\Om$.
The Poincar\'e inequality for $\Np_0$ (e.g.\
Corollary~5.54 in~\cite{BBbook})
then yields
\[
\int_\Om \phi^p\,d\mu \le C_\Om \int_\Om |D\phi|^p\,d\mu = 0.
\]
Hence $\phi=0$ a.e.\ in $\Om$, i.e.\ $u_1\le u_2$ a.e.\ in $\Om$.
As $u_1$ and $u_2$ are lsc-regularized, it follows
that $u_1 \le u_2$ everywhere in $\Om$.
\end{proof}

\begin{remark} \label{rmk-Ch-pot}
Note that if $u_E$ is the Cheeger capacitary potential of $E$ in $\Om$,
given by Definition~\ref{deff-Ch-pot}, then
$u_E$ is the lsc-regularized solution of the
$\K_{\psi,0}(\Om)$-Cheeger
obstacle
problem, where $\psi=1$ in $E$ and $\psi=-\infty$ otherwise.
Hence, for every $\phi\in\Np_0(\Om\setm E)$ and every $t>0$, the function
$u_E+ t\phi\in\K_{\psi,0}(\Om)$ and thus
\[
0 \le \int_\Om ( |Du_E + t D\phi|^p - |Du_E|^p )\,d\mu.
\]
Dividing by $t$ and letting $t\to0$ shows that
\begin{equation}  \label{eq-cap-pot-int-ge0}
\int_\Om |Du_E|^{p-2} Du_E \cdot D\phi \,d\mu \ge0,
\end{equation}
see~(2.8) in Mal\'y--Ziemer~\cite{MZ}.
Applying this also to $-\phi$ shows that equality must hold
in~\eqref{eq-cap-pot-int-ge0}.
Consequently, the measure $\nu_E=Tu_E$ satisfies
\begin{equation}   \label{eq-phi-nu-0}
\int_\Om \phi\,d\nu_E =0
\quad \text{for every }\phi\in\Np_0(\Om\setm E).
\end{equation}
\end{remark}

We will need the following lemma when proving the Cartan
  property (Theorem~\ref{thm-cartan}).
Later, in Theorem~\ref{thm-bootstrap}, we will generalize
this lemma to quasiopen sets and as a consequence obtain that
the measure $\nu_E$ is supported on the fine boundary $\bdyp E$;
that it is supported on the boundary $\bdy E$ is well known.

\begin{lem}  \label{lem-v-u-E-nu-equal}
Let $E\subset\Om$ be such that $\cp(E,\Om)<\infty$ and let $u_E$ be the
Cheeger
capacitary potential of $E$ in $\Om$, with the corresponding Radon measure
$\nu_E=Tu_E$.
If $G\subset\Om$ is open and $v\in\Np(\Om)$ is bounded and
such that $v=1$ q.e.\ in
$G\cap E$ then
\begin{equation}   \label{eq-equal-v-u-E-nu}
\int_G v\,d\nu_E = \int_G u_E\,d\nu_E.
\end{equation}
In particular, $\nu_E(G)=\int_G u_E\,d\nu_E$, and
if $\Cp(G\cap E)=0$ then
$\nu_E(G)=0$.
\end{lem}

\begin{proof}
For every $\eta\in\Lip_c(G)$ with $0\le\eta\le1$ we have
$\eta(v-u_E)\in\Np_0(\Om\setm E)$.
Thus, \eqref{eq-phi-nu-0} yields that
\[
\int_{G} \eta(v-u_E)\,d\nu_E = 0.
\]
Since $v-u_E$ and $G$ are bounded,
dominated convergence and letting $\eta\nearrow \chi_{G}$
imply~\eqref{eq-equal-v-u-E-nu}.
For the last part,
apply this to $v=1$ and $v=0$ respectively.
\end{proof}

\section{Thinness and the fine topology}
\label{sect-fine-cont}

We now define the fine topological notions which are central
in this paper.

\begin{deff}\label{deff-thinness}
A set $E\subset X$ is  \emph{thin} at $x\in X$ if
\begin{equation}   \label{deff-thin}
\int_0^1\biggl(\frac{\cp(E\cap B(x,r),B(x,2r))}{\cp(B(x,r),B(x,2r))}\biggr)^{1/(p-1)}
     \frac{dr}{r}<\infty.
\end{equation}
A set $U\subset X$ is \emph{finely open} if
$X\setminus U$ is thin at each point $x\in U$.
\end{deff}

It is easy to see that the finely open sets give rise to a
topology, which is called the \emph{fine topology}.
Every open set is finely open, but the converse is not true in general.

In the definition of thinness,
we make the convention that the integrand
is 1 whenever $\cp(B(x,r),B(x,2r))=0$.
This happens e.g.\ if $X=B(x,2r)$,
but never if $r < \frac{1}{2}\diam X$.
Note that thinness is a local property.
Because of \eqref{eq-cpp-cp}, thinness
can equivalently be defined using
the Cheeger variational capacity $\cpp$.

\begin{deff}
A function $u : U \to \eR$, defined on a finely open set $U$, is
\emph{finely continuous} if it is continuous when $U$ is equipped with the
fine topology and $\eR$ with the usual topology.
\end{deff}

Since every open set is finely open, the fine topology
generated by the finely open\index{finely!open} sets is finer than the metric topology.
In fact, it is the coarsest topology making all (Cheeger) superharmonic functions
finely continuous, by J.~Bj\"orn~\cite[Theorem~4.4]{JB-pfine},
Korte~\cite[Theorem~4.3]{korte08}
and Bj\"orn--Bj\"orn--Latvala~\cite[Theorem~1.1]{BBL1}.
See \cite[Section~11.6]{BBbook} and \cite{BBL1}
for further discussion on thinness and the fine topology.

\section{The Cartan, Choquet and Kellogg properties}
\label{sect-Cartan}

We start this section by proving
the Cartan property (Theorem~\ref{thm-cartan}).
The proof combines arguments
in Kilpel\"ainen--Mal\'y~\cite[p.~155]{KiMa} with those in Section~2.1.5
in Mal\'y--Ziemer~\cite{MZ}.
As in \cite{KiMa}, the pointwise estimate \eqref{eq-est-u(x0)} is essential here.
However, to obtain the estimate $\nu_k(B_j)\le \capp (E_j,B_{j-1})$, in
\cite{KiMa} they use the dual characterization of capacity as the
supremum of measures on $E_j$ with potentials bounded by 1.
A similar estimate follows also from Theorem~2.45 in \cite{MZ}.
Here we instead use a direct derivation of
$\nu_k(B_j)\le \capp (E_j,B_{j-1})$ based on
\eqref{eq-deff-Tu}, Remark~\ref{rmk-Ch-pot} and
Lemma~\ref{lem-v-u-E-nu-equal}.

\begin{proof}[Proof of Theorem~\ref{thm-cartan}]
By Lemma~4.7
in Bj\"orn--Bj\"orn--Latvala~\cite{BBL1},
we may assume that $E$ is open.
Let $B_j = B(x_0,r_j)$, $r_j = 2^{-j}$,
$E_j = E\cap B_j$
and $u_{j}$ be the Cheeger capacitary potential
of $E_j$ with respect to $B_{j-1}$, $j=1,2,\ldots\ $.
As $E_j$ is open, we have $u_{j}=1$ in $E_j$.
Let $k\geq 1$ be an
integer to be specified later,
but so large that $\diam B_k < \frac{1}{6} \diam X$,
and let
$\nu_k=Tu_k$ be the Radon measure in $\Np_0(B_{k-1})'$,
given by Proposition~\ref{prop-supersol-Radon}.

Since $u_{k}=1$ in $E_k$, it  remains to show that
$u_{k}(x_0) < 1$ for some $k$.
By Remark~5.4 in Kinnunen--Martio~\cite{KiMa02}
(or Proposition~8.24 in \cite{BBbook}), $x_0$ is a Lebesgue point of $u_k$. Hence,
Proposition~4.10 in
Bj\"orn--MacManus--Shanmugalingam~\cite{BMS} shows that
\begin{equation}   \label{eq-est-u(x0)}
u_{k}(x_0)
\leq c \biggl( \vint_{B_k}u_{k}^p\,d\mu \biggr)^{1/p} +
c\sum_{j=k-1}^\infty\biggl(r_j^p\frac{\nu_k(B_j)}{\mu(B_j)}\biggr)^{1/(p-1)}.
\end{equation}
The first term in the right-hand side can be estimated using
the Sobolev inequality \cite[Theorem~5.51]{BBbook}
and the fact that
$\cp(B_k,B_{k-1}) \simeq r_k^{-p}\mu(B_k)$ (by
Lemma~3.3 in J.~Bj\"orn~\cite{Bj} or Proposition~6.16 in \cite{BBbook})
as
\begin{align}    \label{eq-est-vint-uk}
\vint_{B_k}u_{k}^p\,d\mu \le \frac{1}{\mu(B_k)} \int_{B_{k-1}} u_k^p\,d\mu
\le \frac{Cr_k^p}{\mu(B_k)} \int_{B_{k-1}} |Du_k|^p \,d\mu
\simeq \frac{\cp(E_k,B_{k-1})}{\cp(B_k,B_{k-1})}.
\end{align}
Here we have also used \eqref{eq-cpp-def} and \eqref{eq-cpp-cp}.

As for the second term in~\eqref{eq-est-u(x0)},
let $v_j$ be the lsc-regularized solution of $Tv_j=\nu_k|_{B_j}$ in
$B_{k-1}$, $j \ge k$. Lemma~\ref{lem-comp-potentials} shows that $v_j\le u_k\le1$
in $B_{k-1}$. Thus, with $v_j$ as a test function in \eqref{eq-deff-Tu},
we have
\begin{equation}  \label{eq-int-vj-le-cap-new}
\int_{B_{k-1}} |D v_j|^p\,d\mu=\int_{B_j} v_j \,d\nu_k \le
\int_{B_j} u_k \,d\nu_k.
\end{equation}
Using Lemma~\ref{lem-v-u-E-nu-equal} (for the first equality below)
and \eqref{eq-deff-Tu} with $u_j$ as a test function
(for the third equality) we obtain that
\begin{align}  \label{eq-int-vj-le-cap}
\int_{B_j} u_k \,d\nu_k
& = \int_{B_j} u_j \,d\nu_k
 =  \int_{B_{k-1}} u_j \,d\nu_k|_{B_j}
 = \int_{B_{k-1}} |D v_j|^{p-2} Dv_j \cdot Du_j \,d\mu \nonumber \\
& \le \biggl( \int_{B_{k-1}} |D v_j|^p\,d\mu \biggr)^{1-1/p}
            \biggl( \int_{B_{k-1}} |D u_j|^p\,d\mu \biggr)^{1/p}.
\end{align}
Together with \eqref{eq-int-vj-le-cap-new} this implies that
\[
\int_{B_{k-1}} |D v_j|^p\,d\mu \le \int_{B_{k-1}} |D u_j|^p\,d\mu
= \cpp(E_j,B_{j-1}),
\]
where $\cpp$ denotes the Cheeger variational capacity.
Inserting this into \eqref{eq-int-vj-le-cap} yields,
\[
\int_{B_j} u_k \,d\nu_k
\le \cpp(E_j,B_{j-1}),
\]
which together with the last part of Lemma~\ref{lem-v-u-E-nu-equal}
and \eqref{eq-cpp-cp} shows that
\begin{align*}
\nu_k(B_j)&= \int_{B_j} u_k \,d\nu_k
\le \cpp(E_j,B_{j-1})
\simeq \cp(E_j,B_{j-1}).
\end{align*}
Hence using $\cp(B_j,B_{j-1}) \simeq r_j^{-p}\mu(B_j)$ again we obtain
\begin{align} \label{eq-est-Wolff-pot}
\sum_{j=k-1}^\infty\biggl(r_j^p\frac{\nu_k(B_j)}{\mu(B_j)}\biggr)^{1/(p-1)}
&\leq C\sum_{j=k-1}^\infty\biggl(\frac{\cp(E_j,B_{j-1})}
           {\cp(B_j,B_{j-1})}\biggr)^{1/(p-1)}.
\end{align}
Since $E$ is thin at $x_0$, both \eqref{eq-est-vint-uk}
and~\eqref{eq-est-Wolff-pot} can be made arbitrarily small by choosing $k$
large enough.
Thus $u_k(x_0)<1$ for large enough $k$.
\end{proof}

We now turn to the proof of the Choquet property
(Theorem~\ref{thm-choquet-intro}).
The following notation
is common in the
literature.
The \emph{base} $b_p E$ of a set $E\subset X$
consists of all points $x\in X$ where $E$ is \emph{thick},
i.e.\ not thin, at $x$.
Using this notation, the Choquet property
can be formulated as follows.

\begin{thm} \label{thm-choquet}
\textup{(Choquet property)}
For any $E\subset X$ and any $\eps >0$ there is an open set $G$ so that
\[
   G \cup b_{p}E = X\quad\textrm{and}\quad C_{p} (E \cap G) < \eps.
\]
\end{thm}

\begin{proof}
Let $\{B_j\}_{j=1}^\infty$ be a countable covering of $X$ by balls
such that every point is covered by arbitrarily small balls. Such a covering exists as $X$ is separable.
Choose $\eps>0$. For each $j$, let $u_j$ be the Cheeger capacitary potential of
$E \cap B_j$ with respect to $2B_j$. Since each $u_j$ is
quasicontinuous, there is an open set $G_j'$ with
$C_p(G_j')< 2^{-j}\eps$ such that the set
\begin{equation} \label{eq-Gj}
G_j:= \{x \in B_j : u_j(x) <1\}\cup G_j'
\end{equation}
is open. We set
\(
    G:= \bigcup_{j=1}^\infty G_j
\)
and will show
that $G \cup b_{p} E = X$.

Choose a point $z\in X \setminus b_{p} E$.
If $\dist(z,E\setm\{z\})>0$, then there is $B_j\ni z$ such that
$B_j \cap E$ is either empty or $\{z\}$.
If $B_j\cap E=\emptyset$, then $u_j\equiv0$.
If $B_j\cap E=\{z\}$, then
the thinness of $E$ at $z$ together with
Proposition~1.3 in Bj\"orn--Bj\"orn--Latvala~\cite{BBL1}
shows that $\Cp(\{z\})=0$, and
hence $u_j\equiv0$ as well.
In both cases we obtain $z\in B_j\subset G_j\subset G$.

We can therefore assume that $z\in\itoverline{E\setm\{z\}}$.
By Theorem~\ref{thm-cartan} (applied to $E \setm \{z\}$),
there is a bounded positive Cheeger superharmonic function
$v$ in an open neighbourhood of $z$ such that
\[
    v(z) < 1< \liminf_{E\ni x\to z} v(x).
\]
Hence we may fix a ball $B_j\ni z$ so that $v$ is Cheeger superharmonic in
$3B_j$ and $v\ge1$ in $B_j\cap E$.
Since $v$ is the lsc-regularized solution of the
$\K_{v,v}(2B_j)$-Cheeger obstacle
problem and $u_j$ is the lsc-regularized solution of the
$\K_{\chi_{B_j\cap E},0}(2B_j)$-Cheeger
obstacle problem, the comparison principle
in Lemma~5.4 in Bj\"orn--Bj\"orn~\cite{BB} (or
Lemma~8.30 in \cite{BBbook}) yields $u_j\le v$ in $2B_j$.
It follows that $u_j(z)<1$, and thus $z \in G_j \subset G$.

It remains to prove that $C_{p} (E \cap G) < \eps$.
For any $j$, we have $u_j \ge 1$ q.e.\ in $E\cap B_j$, and thus
\eqref{eq-Gj} implies
\[
    C_{p} (E \cap G_j) \le C_{p} ( \{x \in E \cap B_j : u_j(x) < 1\}) +
    C_{p} (G_j') = \Cp(G_j') < 2^{-j} \eps.
\]
By the countable subadditivity of the capacity we obtain $C_{p} (E \cap G) < \eps$.
\end{proof}

As a consequence of the Choquet property we can now deduce
Corollary~\ref{cor-kellogg-intro}. Because of Remark~\ref{twokelloggs}
below, we find the name fine Kellogg property natural.

\begin{cor}\textup{(Fine Kellogg property)}\label{cor-kellogg}
For any $E\subset X$ we have
\[
C_p(E\setminus b_pE)=0.
\]
\end{cor}

\begin{proof} For every $\eps>0$, Theorem~\ref{thm-choquet} provides us with an open set
$G$ such that $G \cup b_{p}E = X$ and $C_{p}(E \cap G) < \eps$.
Then $E\setminus b_pE \subset E\cap G$, and therefore $C_p(E\setminus b_pE)<\eps$.
Letting $\eps\to0$ concludes the proof.
\end{proof}

\begin{remark}\label{twokelloggs}
Let $\Omega\subset X$ be a bounded open set with $\Cp(X \setm \Om)>0$.
Choosing $E=X \setm \Omega$ in Corollary~\ref{cor-kellogg} gives
\begin{equation}\label{eq-two-kell}
C_p(\partial\Omega\setminus b_p(X \setm \Omega))\le C_p((X \setm \Omega)\setminus b_p(X \setm \Omega))=0.
\end{equation}
On the other hand,
a boundary point $x_0\in\partial\Omega$ is regular
(both for \p-harmonic functions defined through upper gradients
and for Cheeger \p-harmonic functions)
whenever
$X \setm \Omega$ is thick at $x_0$,
by the sufficiency part of the Wiener criterion,
see Bj\"orn--MacManus--Shanmugalingam~\cite{BMS} and
J.~Bj\"orn~\cite{JB-Matsue}, \cite{JB-pfine} (or Theorem~11.24 in \cite{BBbook}).
Hence \eqref{eq-two-kell} yields that
 the set of irregular boundary points of $\Om$ is of capacity zero.
This result was obtained
by a different method
(and called the Kellogg property) in
Bj\"orn--Bj\"orn--Shanmugalingam~\cite[Theorem~3.9]{BBS}.

To clarify that the above
proof of the Kellogg property is not using
circular reasoning let us explain
how the results we use here
are obtained in \cite{BBbook}.
Here we only need results up to Chapter~9 therein plus the results
in Sections~11.4 and~11.6.
They in turn only rely on results
up to Chapter~9 plus the implication (b) $\imp$ (a) in Theorem~10.29,
which can easily be obtained just using comparison.
Hence we are not relying on the Kellogg property obtained
in Section~10.2 in \cite{BBbook}.
\end{remark}

\section{Finely open and quasiopen sets}
\label{sect-finely-quasi}

We start this section by
using the Choquet property to prove Theorem~\ref{thm-finelyopen-quasiopen},
i.e.\ we  characterize quasiopen sets
and quasicontinuity by means of the
corresponding fine topological notions.
We then proceed by giving several immediate applications of this characterization.

Note that if $\Cp(\{x\})=0$,
then $\{x\}$ is quasiopen, but not finely open.
Thus the zero capacity set in  Theorem~\ref{thm-finelyopen-quasiopen}\,(a)
cannot be dropped.

\begin{proof}[Proof of Theorem~\ref{thm-finelyopen-quasiopen}]

 (a) That each quasiopen set $U$ is of the form $U=V \cup E$ for some
  finely open set $V$ and for a set $E$ of capacity zero,
was recently shown in
Bj\"orn--Bj\"orn--Latvala~\cite[Theorem~4.9]{BBL1}.

For the converse, assume that $U=V \cup E$, where $V$ is finely open and $\Cp(E)=0$.
Let $\eps>0$. By the Choquet property (Theorem~\ref{thm-choquet}), applied to $X \setm V$,
there is an open set $G$ such that
\[
     G \cup b_p(X \setm V) = X
     \quad \text{and} \quad
     \Cp(G \setm V) < \eps.
\]
The capacity $\Cp$ is an outer capacity,
by Corollary~1.3 in Bj\"orn--Bj\"orn--Shan\-mu\-ga\-lin\-gam~\cite{BBS5}
(or Theorem~5.31 in \cite{BBbook}),
so there is an open set $\widetilde G \supset (G \setm V) \cup E$ such that
$\Cp(\widetilde G)< \eps$.
Since $V$ is finely open, we have $V\subset X\setm b_p(X\setm V)\subset G$, and thus
$
    U \cup \widetilde G = V \cup \widetilde G =G \cup \widetilde G
$
is open, i.e.\ $U$ is quasiopen.

(b) If $u$ is quasicontinuous,
then it is finite q.e., by definition,
and finely continuous q.e., by
Theorem~4.9 in Bj\"orn--Bj\"orn--Latvala~\cite{BBL1}.

Conversely, assume that there is a set $Z$ with $\Cp(Z)=0$ such that $u$ is finite and finely continuous on $V:=U\setminus Z$.
By (a), we can assume that $V$ is finely open.
Let $\eps>0$ and let $\{(a_j,b_j)\}_{j=1}^\infty$
be an enumeration of all open intervals with rational endpoints and set
\[
   V_j:=\{x \in V:a_j<u(x)<b_j\}.
\]
By the fine continuity of $u$, the sets $V_j$ are finely open.
Hence by (a), $V_j$ are quasiopen, and thus
there are open sets $G_j$ and $G_{U}$ with
$\Cp(G_j)<2^{-j}\eps$ and $\Cp(G_{U})<\eps$ such that $V_j\cup G_j$ and
$U\cup G_{U}$ are open. Also, as $\Cp$ is an outer capacity,
there is an open set $G_Z\supset Z$ with $\Cp(G_Z)<\eps$.
Then
\[
G:=G_Z\cup G_{U}\cup\bigcup_{j=1}^{\infty}G_j
\]
is open, $\Cp(G)<3\eps$, and
$u|_{U \setm G}$ is continuous since $V_j\cup G$ are open sets.
\end{proof}

Theorem~\ref{thm-finelyopen-quasiopen} leads directly to the following
improvements of the results in Bj\"orn--Bj\"orn~\cite{BBnonopen}.

\begin{cor}  \label{cor-finely-open=>p-path}
\label{lem-quasiopen-measurable}
Every finely open set is quasiopen, measurable and \p-path open.
\end{cor}

A set $U$ is \emph{\p-path open} if for
\p-almost every curve $\ga:[0,l_\ga]\to X$, the set
$\ga^{-1}(U)$ is (relatively) open in $[0,l_\ga]$.

\begin{proof}
By Theorem~\ref{thm-finelyopen-quasiopen}\,(a)
every finely open set is quasiopen.
Hence the result follows from Remark~3.5
in Shanmugalingam~\cite{Sh-harm} and
Lemma~9.3 in \cite{BBnonopen}.
\end{proof}

An important consequence
is that the restriction
of a minimal \p-weak upper gradient to a finely open set
remains minimal.
This was shown for measurable \p-path open
sets in \cite[Corollary~3.7]{BBnonopen}.
We restate this result, in view of Corollary~\ref{cor-finely-open=>p-path}.
In order to do so in full generality, we need to introduce some
more notation.

We define the Dirichlet space
\[
   \Dp(X)=\{u : u \text{ is measurable and $u$ has an upper gradient
     in }   L^p(X)\}.
\]
As with $\Np(X)$ we assume that functions
in $\Dp(X)$  are defined everywhere (with values in $\eR:=[-\infty,\infty]$).
For a measurable set $E\subset X$,
the spaces $\Dp(E)$ and $\Dploc(E)$ are defined similarly.
For $u \in \Dploc(E)$ we denote the minimal
\p-weak upper gradient of $u$ taken with  $E$
as the underlying space by $g_{u,E}$. Its existence is guaranteed by Theorem~2.25 in~\cite{BBbook}.

\begin{cor} \label{cor-minimal-restrict}
Let $U$ be quasiopen and $u \in \Dploc(X)$.
Then $g_{u,U}=g_u$ a.e.\ in $U$.
In particular this holds if $U$ is finely open.
\end{cor}

\begin{proof} By Remark~3.5
in Shanmugalingam~\cite{Sh-harm} and
Lemma~9.3
in  \cite{BBnonopen} every quasiopen set is
\p-path open and measurable, whereas
Theorem~\ref{thm-finelyopen-quasiopen}\,(a)
shows that
every finely open set is quasiopen.
Hence the result follows from
Corollary~3.7 in \cite{BBnonopen}.
\end{proof}

In~\cite{BBnonopen}, the fine topology turned out to be important for obstacle
problems on nonopen measurable sets, i.e.\ when minimizing the \p-energy integral
\begin{equation}  \label{eq-min-p-energy}
\int_E g_{u,E}\,d\mu
\end{equation}
on an arbitrary bounded measurable set $E$ among all functions 
\[
u \in \K_{\psi_1,\psi_2,f}(E):=\{v \in \Dp(E) : v-f \in \Np_0(E)
   \text{ and } \psi_1 \le v \le \psi_2 \ \text{q.e. in } E\}.
\]
Knowing that finely open sets are measurable and \p-path open,
we are now able to improve
and simplify some of the results therein.
We summarize these improvements in the following theorem,
which follows directly from
\cite[Theorems~1.2 and 8.3, and Corollaries~3.7 and~7.4]{BBnonopen}
and Corollary~\ref{cor-finely-open=>p-path}.
We denote the fine interior of $E$ by $\fineint E$.

\begin{thm}
Let $E \subset X$ be a bounded measurable set such that $\Cp(X \setm E)>0$,
and let $f\in\Dp(E)$ and $\psi_j:E \to \eR$, $j=1,2$,  be such that
$\K_{\psi_1,\psi_2,f}(E)\ne \emptyset$. Also let $E_0=\fineint E$.

Then $\K_{\psi_1,\psi_2,f}(E) = \K_{\psi_1,\psi_2,f}(E_0)$, and the solutions
of the minimization problem for~\eqref{eq-min-p-energy}
with respect to $\K_{\psi_1,\psi_2,f}(E)$ and $\K_{\psi_1,\psi_2,f}(E_0)$ coincide.
Moreover, $g_{u,E_0}=g_{u,E}$ a.e.\ in $E_0$ and
if $\mu(E\setm E_0)=0$ then also the \p-energies associated
with these two minimization problems coincide.

Furthermore, if $f\in\Dp(\Om)$ for some open set\/ $\Om\supset E$, then
$g_{u,E_0}=g_{u,E}=g_u$ a.e.\ in $E_0$ and
the above solutions coincide with the solutions of the corresponding
$\K_{\psi'_1,\psi'_2,f}(\Om)$-obstacle problem,
where $\psi'_j$ is the extension of $\psi_j$ to $\Om\setm E$ by $f$, $j=1,2$.
\end{thm}

We also obtain the following consequence of Lemma~3.9
and Theorem~7.3 in \cite{BBnonopen}, which generalizes
Theorem~2.147 and Corollary~2.162 in Mal\'y--Ziemer~\cite{MZ}
to metric spaces and to arbitrary sets.
See also Remark~2.148 in~\cite{MZ} for another description of
$W^{1,p}_0(\Om)$ in $\R^n$.

\begin{prop}   \label{prop-char-Np0-qe-bdry}
\textup{(Cf.\ Proposition~9.4 in \cite{BBnonopen}.)}
Let $E\subset X$ be arbitrary and $u\in\Np(\clEp)$,
where $\clEp$ is the fine closure of $E$.
Then $u\in\Np_0(E)$ if and only if $u=0$ q.e.\ on the fine boundary
$\bdy_pE :=\itoverline{E}^p\setm\fineint E$ of $E$.
\end{prop}

\section{Support of capacitary measures}
\label{sect-supp-cap-meas}

We can now bootstrap
Lemma~\ref{lem-v-u-E-nu-equal} to quasiopen sets
and in particular show that the capacitary measure $\nu_E$ only charges the fine
boundary $\bdry_p E:= \clEp \setm \fineint E$ of $E$, where $\clEp$ is the fine closure of $E$.
This observation seems to be new even in unweighted $\R^n$.

\begin{thm}  \label{thm-bootstrap}
Let $\Om$ be a nonempty bounded open set with $\Cp(X \setm \Om)>0$.
Let $E\subset\Om$, $u_E$ and $\nu_E=Tu_E$ be as in
Lemma~\ref{lem-v-u-E-nu-equal}.
Let $U\subset\Om$ be quasiopen and $v\in\Np(\Om)$.
Then the following are true\/\textup{:}
\begin{enumerate}
\item \label{bootstrap-a}
If $u \in \Np(\Om)$ and either $u$ is bounded from below or belongs to $L^1(\nu_E)$,
and $u=v$ q.e.\ in $U \cap E$, then
\begin{equation} \label{eq-bootstrap}
  \int_U u \,d\nu_E = \int_U v\,d\nu_E.
\end{equation}
\item \label{bootstrap-b}
If $v=1$ q.e.\ in $U\cap E$, then
\[
\nu_E(U)=\int_U v \,d\nu_E = \int_U u_E\,d\nu_E.
\]
\item \label{bootstrap-c}
If $\Cp(U\cap E)=0$, then
$\nu_E(U)=0$.
\end{enumerate}
\end{thm}

\begin{remark}
We shall see in Corollary~\ref{cor-bootstrap} below that the set $U\cap E$ in
\ref{bootstrap-a}, \ref{bootstrap-b} and \ref{bootstrap-c}
above can be replaced  by $U\cap\bdry_p E$ and that
the assumption $v\in\Np(\Om)$ in Theorem~\ref{thm-bootstrap} can be omitted
in that case.
\end{remark}

To prove Theorem~\ref{thm-bootstrap} we need the following quasi-Lindel\"of principle,
whose proof in unweighted $\R^n$ is given in Theorem~2.3 in
Heinonen--Kilpel\"ainen--Mal\'y~\cite{HeKiMaly}.
This proof, which relies on the fine Kellogg property, extends to metric spaces,
see Bj\"orn--Bj\"orn--Latvala~\cite{BBL3}.

\begin{thm} \label{thm-quasiLindelof}
\textup{(Quasi-Lindel\"of principle)}
 For each family $\mathcal V$ of finely open sets there is a countable subfamily $\mathcal V'$ such that
\[
\Cp\biggl(\bigcup_{V\in\mathcal V}V\setminus \bigcup_{V'\in\mathcal V'}V'\biggr)=0.
\]
\end{thm}

We also need the following lemmas.

\begin{lem} \label{lem-finely-open-x}
Let $U$ be finely open and let $x_0\in U$.
Then there exists a finely open set $V\Subset U$ containing $x_0$
and a function $v \in \Np_0(U)$ such that $v=1$ on $V$
and $0 \le v \le 1$ everywhere.
\end{lem}

\begin{proof}
Since $U$ is finely  open, $E:=X \setm U$ is thin at $x_0$.
By the Cartan property (Theorem~\ref{thm-cartan}),
there are  a ball $B \ni x_0$ and a lower semicontinuous
finely continuous
$u \in \Np(B)$ such that $0 \le u \le 1$ in $B$,
$u(x_0)<1$ and $u=1$ in
$E\cap B$.
Let $\eta\in\Lipc(B)$ be
such that $0\le\eta\le1$
in $B$ and $\eta=1$ in $\tfrac12 B$.
Then $w:=\eta(1-u)\in\Np_0(U)$ is upper semicontinuous and finely continuous
in $X$ and $w(x_0)=1-u(x_0)>0$.
Let $v=\min\{1,2w/w(x_0)\} \in \Np_0(U)$ and
$V=\bigl\{x \in U : w(x)> \tfrac{1}{2} w(x_0)\bigr\}$.
The fine continuity and upper semicontinuity of $w$ imply that $V$ is
finely open and $V\Subset U$. Moreover $x_0 \in V$ and $v=1$ on $V$.
\end{proof}

\begin{lemma} \label{lem-quasiopen-Borel}
Let $U\subset X$ be quasiopen. Then
\begin{equation} \label{eq-Borel}
U=W_1\cup E_1=W_2\setminus E_2,
\end{equation}
where $W_1$ and $W_2$ are Borel sets and $E_1$ and $E_2$ are of capacity zero.
Moreover, we may choose $W_1$ to be of type $F_\sigma$ and $W_2$ to be of type $G_\delta$.
\end{lemma}

Not all finely open sets are Borel.
Let for instance $V=G \setm A$, where $G$ is open and $A \subset G$
is a non-Borel set with $\Cp(A)=0$.
Then $V$ is a non-Borel finely open set.
To be more specific, we may let $A \subset G \subset \R^n$ be any
non-Borel set of Hausdorff dimension $< n-p$.

\begin{proof}
By definition, for each $j=1,2,\ldots$ there is an
 open set $G_j$ such that $U\cup G_j$ is open and $\Cp(G_j)<1/j$. Then
\[
\begin{split}
U&=\biggl(U\setminus\bigcap_{j=1}^{\infty}G_j\biggr)\cup
   \biggl(U\cap\bigcap_{j=1}^{\infty}G_j\biggr)
   =\bigcup_{j=1}^{\infty}(U\setminus G_j)
    \cup \bigcap_{j=1}^{\infty} (U\cap G_j)\\
&=\bigcup_{j=1}^{\infty}((U\cup G_j)\setminus G_j)\cup \bigcap_{j=1}^{\infty}(U\cap G_j) =: W_1 \cup E_1.
\end{split}
\]
The second equality in \eqref{eq-Borel}
follows by choosing $W_2=\bigcap_{j=1}^{\infty}(U \cup G_j)$ and
$E_2=W_2\setminus U$. The last two claims follow from the choices above.
\end{proof}

\begin{proof}[Proof of Theorem~\ref{thm-bootstrap}]

By Theorem~\ref{thm-finelyopen-quasiopen}, we can find a finely open set $V\subset U$ such that $\Cp(U\setm V)=0$.
For every $x \in V$, Lemma~\ref{lem-finely-open-x}
provides us with a finely open set $V_x\Subset V$ containing $x$
and a function $v_x \in \Np_0(V)$ such that $v_x=1$ on $V_x$
and $0 \le v_x \le 1$ everywhere.
By the  quasi-Lindel\"of principle,
and the fact that $\Cp(U\setm V)=0$, we can out of these
choose
$V_j=V_{x_j}$ and $v_j=v_{x_j}$, $j=1,2,\ldots$,
so that $U=\bigcup_{j=1}^\infty V_j \cup Z$,
where $\Cp(Z)=0$.
For $k=1,2,\ldots$, set
\[
\eta_k= \chi_{X \setm Z} \max_{j=1,2,\ldots,k} v_j\in\Np_0(U).
\]

Since $\nu_E$ is a complete Borel measure which,
by Lemma~\ref{lem-v-u-E-nu-equal}
(or Lemma~3.8 in Bj\"orn--MacManus--Shan\-mu\-ga\-lin\-gam~\cite{BMS}),
is absolutely continuous with respect to the capacity $\Cp$,
it follows from Lemma~\ref{lem-quasiopen-Borel} that $U$ is $\nu_E$-measurable
and $\nu_E(Z)=0$. We are now ready to prove \ref{bootstrap-a}--\ref{bootstrap-c}.

\ref{bootstrap-a} First, assume that $u$ and $v$ are bounded.
Then $\eta_k(u-v)\in\Np(U)$, by Lemma~\ref{lem-product},
and Lemma~2.37 in Bj\"orn--Bj\"orn~\cite{BBbook}
shows that $\eta_k(u-v)\in \Np_0(U)$.
Since $u=v$ q.e.\ in $U \cap E$,
it follows that
$\eta_k(u-v)\in \Np_0(U\setm E)$.
Hence \eqref{eq-phi-nu-0} yields that
\[
\int_{U}\eta_k(u-v)\,d\nu_E = 0.
\]
Since $\eta_k\nearrow \chi_{U\setm Z}$ in $U$, dominated convergence
and the fact that $\nu_E(Z)=0$ imply that
\[
\int_{U} (u-v)\,d\nu_E=\int_{U \setm Z} (u-v)\,d\nu_E=0,
\]
and \eqref{eq-bootstrap} follows.

Next, assume that $u$ and $v$ are bounded from below.
Then, by monotone convergence and the bounded case,
\[
  \int_U u \,d\nu_E =
  \lim_{k \to \infty} \int_U \min\{u,k\} \,d\nu_E
  =   \lim_{k \to \infty} \int_U \min\{v,k\} \,d\nu_E
   = \int_U v\,d\nu_E.
\]
Finally, applying this to the positive and negative parts of $u$ and $v$ gives
\[
\int_U u_\limplus \,d\nu_E = \int_U v_\limplus \,d\nu_E
\quad \text{and} \quad
\int_U u_\limminus \,d\nu_E = \int_U v_\limminus \,d\nu_E,
\]
and hence
\[
\int_U u \,d\nu_E = \int_U u_\limplus \,d\nu_E - \int_U u_\limminus \,d\nu_E
= \int_U v_\limplus \,d\nu_E - \int_U v_\limminus \,d\nu_E = \int_U v\,d\nu_E,
\]
where the assumptions on $u$ guarantee that the subtractions are well
defined (i.e.\ not $\infty-\infty$).

\ref{bootstrap-b}
By applying \ref{bootstrap-a} to $u=u_E$ and $v$ we have
$\int_U v \,d\nu_E = \int_U u_E\,d\nu_E$.
Choosing $v \equiv 1$ yields
$
\nu_E(U) = \int_U u_E\,d\nu_E$.

\ref{bootstrap-c}
This follows by applying \ref{bootstrap-b} to $v\equiv 0$.
\end{proof}

\begin{cor} \label{cor-bootstrap}
Let $\Om$, $E$, $u_E$ and $\nu_E$ be as in Theorem~\ref{thm-bootstrap}.
Then
\[
   \nu_E(\Om \setm \bdry_p E)=  0,
\]
i.e.\
$\nu_E$ is supported on the fine boundary
$\bdry_p E:= \clEp \setm \fineint E$ of $E$.
\end{cor}

\begin{proof}
First, the fine exterior $V=\Om\setm\clEp$ is finely open and
$V\cap E=\emptyset$, whence $\nu_E(V) = 0$ by
Theorem~\ref{thm-bootstrap}\,\ref{bootstrap-c}.

Next, the fine interior $E_0:=\fineint E$ is finely open and as in the proof of
Theorem~\ref{thm-bootstrap} we can use the quasi-Lindel\"of principle to
find nonnegative $\eta_k\in\Np_0(E_0)$ such that
\[
\eta_k\nearrow \chi_{E_0\setm Z}
\quad \text{as } k \to \infty,
\]
where $\Cp(Z)=0$.
Since $u_E=1$ q.e.\ in $E$, we have $Du_E=0$ a.e.\ in $E$ and hence 
by \eqref{eq-deff-Tu}
\[
\int_\Om \eta_k \,d\nu_E = \int_\Om |Du_E|^{p-2} Du_E \cdot D\eta_k \,d\mu =0.
\]
Dominated convergence then shows that
$\nu_E(E_0 \setm Z)=0$.
Since $\nu_E(Z)=0$ by Lemma~\ref{lem-v-u-E-nu-equal}
(or Lemma~3.8 in \cite{BMS}),
the proof is complete.
\end{proof}

\begin{cor}  \label{cor-bootstrap2}
Let $\Om$, $E\subset\Om$, $u_E$ and $\nu_E=Tu_E$ be as in
Lemma~\ref{lem-v-u-E-nu-equal}.
Let $U\subset\Om$ be quasiopen. 
Then the following are true\/\textup{:}
\begin{enumerate}
\item \label{bootstrap2-a}
If $u$ is a function on $\Omega$ such that $\int_ {U\cap \bdry_p E}u\,d\nu_E$ 
is well-defined 
and $v$ is a function on $U$ such that $v=u$ q.e.\ in $U \cap \bdry_p E$, then
\[
  \int_U v \,d\nu_E = \int_U u\,d\nu_E.
\]
\item \label{bootstrap2-b}
If $v=1$ q.e.\ in $U\cap \bdry_p E$, then
\[
\nu_E(U)=\int_U v \,d\nu_E = \int_U u_E\,d\nu_E.
\]
\item \label{bootstrap2-c}
If $\Cp(U\cap \bdry_p E)=0$, then $\nu_E(U)=0$.
\end{enumerate}
\end{cor}

\begin{proof}
\ref{bootstrap2-c}
This follows directly from Corollary~\ref{cor-bootstrap} and the fact that $\nu_E$
is absolutely continuous with respect to the capacity $\Cp$
(by Lemma~\ref{lem-v-u-E-nu-equal}).

\ref{bootstrap2-a}
By Corollary~\ref{cor-bootstrap}
and the
absolute continuity of $\nu_E$ with respect to $C_p$ 
again, we see that
\[
  \int_U v \,d\nu_E = \int_{U\cap\bdry_p E} v\,d\nu_E = \int_{U\cap\bdry_p E} u\,d\nu_E =
\int_U u \,d\nu_E.
\]

\ref{bootstrap2-b}
This follows from \ref{bootstrap2-a} by choosing $u\equiv 1$ and $u=u_E$,
respectively.
\end{proof}

We end with a simple example showing that the fine boundary can be much smaller than the metric boundary. A much more involved example in the same spirit is given in Section~9 in Bj\"orn--Bj\"orn~\cite{BBnonopen}.

\begin{example} Let $B$ be an open ball in $\R^n$, $1<p\le n$,
and let $E=B\setminus\Q^n$.
The set $E$ is finely open and has fine closure $\clEp=\itoverline{B}$.
Hence $\bdry_p E= \bdy B \cup (B \cap \Q^n)$, while $\bdy E=\itoverline{B}$.
\end{example}

\end{document}